\documentclass[9pt,twoside,reqno]{amsart}
\usepackage{amssymb}
\usepackage{amsmath}
\usepackage{amsfonts}

\newtheorem{theorem}{Theorem}
\theoremstyle{plain}

\newtheorem{definition}{Definition}

\newtheorem{lemma}{Lemma}

\numberwithin{equation}{section}
\input{tcilatex}

\begin{document}
\title[$g-$CONVEX DOMINATED FUNCTIONS]{ON THE CO-ORDINATED $g-$CONVEX
DOMINATED FUNCTIONS}
\author{M. Emin \"{O}zdemir$^{\blacksquare }$}
\address{$^{\blacksquare }$Atat\"{u}rk University, K.K. Education Faculty,
Department of Mathematics, 25240, ERZURUM TURKEY}
\email{emos@atauni.edu.tr}
\author{Alper Ekinci$^{\spadesuit }$}
\address{$^{\spadesuit }$A\u{g}r\i\ \.{I}brahim \c{C}e\c{c}en University,
Faculty of Science and Letters, Department of Mathematics, 04100 A\u{G}RI
TURKEY}
\email{alperekinci@hotmail.com}
\author{Ahmet Ocak Akdemir$^{\spadesuit }$}
\email{ahmetakdemir@agri.edu.tr}

\begin{abstract}
In this study, we define $g-$convex dominated functions on the co-ordinates
and prove some Hadamard-type, Fejer-type inequalities for this new class of
functions. We also give some results related to the functional $H.$
\end{abstract}

\maketitle

\section{INTRODUCTION}

A function $f:I\subseteq 
\mathbb{R}
\rightarrow 
\mathbb{R}
$ is said to be convex function on $I$ if the inequality%
\begin{equation*}
f(tx+(1-t)y)\leq tf(x)+(1-t)f(y),
\end{equation*}%
holds for all $x,y\in I$ and $t\in \lbrack 0,1].$ The classical
Hermite-Hadamard inequality gives us an estimate of the mean value of a
convex function $f:I\subseteq 
\mathbb{R}
\rightarrow 
\mathbb{R}
$ which is well-known in the literature as following;%
\begin{equation*}
f\left( \frac{a+b}{2}\right) \leq \frac{1}{b-a}\dint\limits_{a}^{b}f(x)dx%
\leq \frac{f(a)+f(b)}{2}.
\end{equation*}

In \cite{5}, Dragomir defined convex functions on the co-ordinates as
following;

\begin{definition}
Let us consider the bidimensional interval $\Delta =[a,b]\times \lbrack c,d]$
in $%
\mathbb{R}
^{2}$ with $a<b,$ $c<d.$ A function $f:\Delta \rightarrow 
\mathbb{R}
$ will be called convex on the co-ordinates if the partial mappings $%
f_{y}:[a,b]\rightarrow 
\mathbb{R}
,$ $f_{y}(u)=f(u,y)$ and $f_{x}:[c,d]\rightarrow 
\mathbb{R}
,$ $f_{x}(v)=f(x,v)$ are convex where defined for all $y\in \lbrack c,d]$
and $x\in \lbrack a,b].$ Recall that the mapping $f:\Delta \rightarrow 
\mathbb{R}
$ is convex on $\Delta $ if the following inequality holds, 
\begin{equation*}
f(\lambda x+(1-\lambda )z,\lambda y+(1-\lambda )w)\leq \lambda
f(x,y)+(1-\lambda )f(z,w)
\end{equation*}%
for all $(x,y),(z,w)\in \Delta $ and $\lambda \in \lbrack 0,1].$
\end{definition}

Every convex function is co-ordinated convex but the converse is not
generally true.

In \cite{5}, Dragomir established the following inequalities of Hadamard's
type for co-ordinated convex functions on a rectangle from the plane $%
\mathbb{R}
^{2}.$

\begin{theorem}
Suppose that $f:\Delta =[a,b]\times \lbrack c,d]\rightarrow 
\mathbb{R}
$ is convex on the co-ordinates on $\Delta $. Then one has the inequalities;%
\begin{eqnarray*}
&&\ f(\frac{a+b}{2},\frac{c+d}{2}) \\
&\leq &\frac{1}{2}\left[ \frac{1}{b-a}\int_{a}^{b}f(x,\frac{c+d}{2})dx+\frac{%
1}{d-c}\int_{c}^{d}f(\frac{a+b}{2},y)dy\right] \\
&\leq &\frac{1}{(b-a)(d-c)}\int_{a}^{b}\int_{c}^{d}f(x,y)dxdy \\
&\leq &\frac{1}{4}\left[ \frac{1}{(b-a)}\int_{a}^{b}f(x,c)dx+\frac{1}{(b-a)}%
\int_{a}^{b}f(x,d)dx\right. \\
&&\left. +\frac{1}{(d-c)}\int_{c}^{d}f(a,y)dy+\frac{1}{(d-c)}%
\int_{c}^{d}f(b,y)dy\right] \\
&\leq &\frac{f(a,c)+f(a,d)+f(b,c)+f(b,d)}{4}.
\end{eqnarray*}%
The above inequalities are sharp.
\end{theorem}

In \cite{8}, Alomari and Darus proved following inequalities of Fejer-type
for co-ordinated convex functions.

\begin{theorem}
Let $f:\left[ a,b\right] \times \left[ c,d\right] \rightarrow 
\mathbb{R}
$ be a co-ordinated convex function. Then the following double inequality
holds:%
\begin{eqnarray}
f\left( \frac{a+b}{2},\frac{c+d}{2}\right) &\leq &\frac{\dint\limits_{a}^{b}%
\dint\limits_{c}^{d}f\left( x,y\right) p\left( x,y\right) dydx}{%
\dint\limits_{a}^{b}\dint\limits_{c}^{d}p\left( x,y\right) dydx}
\label{fejer co} \\
&\leq &\frac{f\left( a,c\right) +f\left( c,d\right) +f\left( b,c\right)
+f\left( b,d\right) }{4}  \notag
\end{eqnarray}%
where $p:\left[ a,b\right] \times \left[ c,d\right] \rightarrow 
\mathbb{R}
$ is positive, integrable and symmetric about $x=\frac{a+b}{2}$ and $y=\frac{%
c+d}{2}.$ The above inequalities are sharp.
\end{theorem}

In \cite{2}, Dragomir \textit{et al.} defined $g-$convex dominated functions
and gave some results related to this functions as following;

\begin{definition}
Let $g:I\rightarrow 
\mathbb{R}
$ be a given convex function on the interval $I$ from $%
\mathbb{R}
.$ The real function $f:I\rightarrow 
\mathbb{R}
$ is called $g-$convex dominated on $I$ if the following condition is
satisfied:%
\begin{eqnarray*}
&&\left\vert \lambda f\left( x\right) +\left( 1-\lambda \right) f\left(
y\right) -f\left( \lambda x+\left( 1-\lambda \right) y\right) \right\vert  \\
&\leq &\lambda g\left( x\right) +\left( 1-\lambda \right) g\left( y\right)
-g\left( \lambda x+\left( 1-\lambda \right) y\right) 
\end{eqnarray*}%
for all $x,y\in I$ and $\lambda \in \left[ 0,1\right] .$ For related
Theorems and classes see:\cite{h1}, \cite{h2}, \cite{h3}.
\end{definition}

\begin{theorem}
Let $g$ be a convex function on $I$ and $f:I\rightarrow 
\mathbb{R}
.$ The following statements are equivalent:

$\left( i\right) $ $\ \ f$ is $g-$convex dominated on $I;$

$\left( ii\right) $ \ The mappings $g-f$ and $g+f$ are convex on $I;$

$\left( iii\right) $ There exist two convex mappings $h,k$ defined on $I$
such that 
\begin{equation*}
f=\frac{1}{2}\left( h-k\right) \text{ and }g=\frac{1}{2}\left( h+k\right) .
\end{equation*}
\end{theorem}

In \cite{5}, Dragomir considered a mapping which closely connected with
above inequalities and established main properties of this mapping as
following:

Now, for a mapping $f:\Delta :=\left[ a,b\right] \times \left[ c,d\right]
\rightarrow 
\mathbb{R}
$ is convex on the co-ordinates on $\Delta ,$ we can define the mapping $H:%
\left[ 0,1\right] ^{2}\rightarrow 
\mathbb{R}
,$%
\begin{equation*}
H(t,s):=\frac{1}{\left( b-a\right) \left( d-c\right) }\dint\limits_{a}^{b}%
\dint\limits_{c}^{d}f\left( tx+(1-t)\frac{a+b}{2},sy+(1-s)\frac{c+d}{2}%
\right) dxdy
\end{equation*}

\begin{theorem}
\lbrack See \cite{5}] Suppose that $f:\Delta \subset 
\mathbb{R}
^{2}\rightarrow 
\mathbb{R}
$ is convex on the co-ordinates on $\Delta =\left[ a,b\right] \times \left[
c,d\right] .$ Then:

(i) The mapping $H$ is convex on the co-ordinates on $\left[ 0,1\right]
^{2}. $

(ii) We have the bounds%
\begin{equation*}
\sup_{(t,s)\in \left[ 0,1\right] ^{2}}H(t,s)=\frac{1}{\left( b-a\right)
\left( d-c\right) }\dint\limits_{a}^{b}\dint\limits_{c}^{d}f\left(
x,y\right) dxdy=H(1,1)
\end{equation*}%
\begin{equation*}
\inf_{(t,s)\in \left[ 0,1\right] ^{2}}H(t,s)=f\left( \frac{a+b}{2},\frac{c+d%
}{2}\right) =H(0,0)
\end{equation*}

(iii) The mapping $H$ is monotonic nondecreasing on the co-ordinates.
\end{theorem}

In this study, we define $g-$ convex dominated functions on $\Delta $ on the
co-ordinates and establish some inequalities of Hadamard-type and Fejer-type
for this class of functions. We also give some results for the functional $%
H. $

\section{MAIN RESULTS}

We will start with the following definition.

\begin{definition}
Let $\Delta =\left[ a,b\right] \times \left[ c,d\right] $ and $g:\Delta
\rightarrow 
\mathbb{R}
$ be a convex function. The real function $f:\Delta \rightarrow 
\mathbb{R}
$ is called a $g-$ convex dominated function on $\Delta $ if%
\begin{eqnarray*}
&&\left\vert \lambda f\left( x,y\right) +\left( 1-\lambda \right) f\left(
z,w\right) -f\left( \lambda x+\left( 1-\lambda \right) z,\lambda y+\left(
1-\lambda \right) w\right) \right\vert \\
&\leq &\lambda g\left( x,y\right) +\left( 1-\lambda \right) g\left(
z,w\right) -g\left( \lambda x+\left( 1-\lambda \right) z,\lambda y+\left(
1-\lambda \right) w\right)
\end{eqnarray*}%
for all $\lambda \in \left[ 0,1\right] $ and $\left( x,y\right) ,\left(
z,w\right) \in \Delta .$
\end{definition}

\begin{definition}
Suppose that $g:\Delta \rightarrow 
\mathbb{R}
$ be a convex function. A function $f:\Delta \rightarrow 
\mathbb{R}
$ is called $g-$convex dominated on co-ordinates if the partial mappings $%
f_{x}:\left[ c,d\right] \rightarrow 
\mathbb{R}
,$ $f_{x}\left( u\right) :=f\left( x,v\right) $ is $g_{x}-$ convex dominated
on $\left[ c,d\right] $ and $f_{y}:\left[ a,b\right] \rightarrow 
\mathbb{R}
,$ $f_{y}\left( w\right) :=f\left( w,y\right) $ is $g_{y}-$ convex dominated
on $\left[ a,b\right] $ where $g_{x}:\left[ c,d\right] \rightarrow 
\mathbb{R}
,$ $g_{x}\left( u\right) :=g\left( x,u\right) $ and $g_{y}:\left[ a,b\right]
\rightarrow 
\mathbb{R}
,$ $g_{y}\left( w\right) :=g\left( w,y\right) $.
\end{definition}

\begin{lemma}
Let $g:\Delta \rightarrow 
\mathbb{R}
$ be a convex function. Every $g-$convex dominated mapping on $\Delta $ is $%
g-$convex dominated on the co-ordinates but the converse may not be
necessarily true$.$
\end{lemma}

\begin{proof}
Suppose that $f:\Delta \rightarrow 
\mathbb{R}
$ is $g-$convex dominated on $\Delta $. Consider $f_{x}:\left[ c,d\right]
\rightarrow 
\mathbb{R}
,$ $f_{x}\left( u\right) :=f\left( x,u\right) $ and $g_{x}:\left[ c,d\right]
\rightarrow 
\mathbb{R}
,$ $g_{x}\left( u\right) :=g\left( x,u\right) .$ Then for all $\lambda \in %
\left[ 0,1\right] $ and $u,w\in \left[ c,d\right] ,$ we have 
\begin{eqnarray*}
&&\left\vert \lambda f\left( x,u\right) +(1-\lambda )f\left( y,w\right)
-f\left( \lambda x+\left( 1-\lambda \right) x,\lambda u+\left( 1-\lambda
\right) w\right) \right\vert  \\
&\leq &\lambda g\left( x,u\right) +(1-\lambda )g\left( y,w\right) -g\left(
\lambda x+\left( 1-\lambda \right) x,\lambda u+\left( 1-\lambda \right)
w\right) 
\end{eqnarray*}%
then we can write%
\begin{eqnarray}
&=&\left\vert \lambda f\left( x,u\right) +(1-\lambda )f\left( y,w\right)
-f\left( x,\lambda u+\left( 1-\lambda \right) w\right) \right\vert 
\label{l-p-1} \\
&\leq &\lambda g\left( x,u\right) +(1-\lambda )g\left( y,w\right) -g\left(
x,\lambda u+\left( 1-\lambda \right) w\right)   \notag
\end{eqnarray}%
it follows that%
\begin{eqnarray}
&&\left\vert \lambda f_{x}\left( u\right) +(1-\lambda )f_{x}\left( w\right)
-f_{x}\left( \lambda u+\left( 1-\lambda \right) w\right) \right\vert 
\label{l-p-2} \\
&\leq &\lambda g_{x}\left( u\right) +(1-\lambda )g_{x}\left( w\right)
-g_{x}\left( \lambda u+\left( 1-\lambda \right) w\right) .  \notag
\end{eqnarray}%
We can show that this results also hold for $f_{y}.$This completes the
proof. Now consider the mappings $g_{0}(x,y)=x+y$ and $f_{0}(x,y)=xy$ where $%
\Delta :=\left[ 0,1\right] \times \left[ 0,1\right] .$ $g_{0}$ is convex on $%
\Delta $ and $f_{0}$ is $g_{0}-$ convex dominated on the co-ordinates but it
can be easily seen that the mapping $f_{0}$ is not $g_{0}-convex$ dominated
on $\Delta .$
\end{proof}

\begin{lemma}
\label{co-g-lemma2} Let $g$ be a convex function on $\Delta $ and $f:\Delta
\rightarrow 
\mathbb{R}
$. If $f$ is $g-$convex dominated on the co-ordinates, the mappings $g-f$
and $g+f$ are convex on the co-ordinates.
\end{lemma}

\begin{proof}
From the fact that $f_{x}$ is $g_{x}-$convex dominated we have 
\begin{eqnarray*}
&&\lambda g_{x}\left( u\right) +(1-\lambda )g_{x}\left( w\right)
-g_{x}\left( \lambda u+\left( 1-\lambda \right) w\right)  \\
&\geq &\left\vert \lambda f_{x}\left( u\right) +(1-\lambda )f_{x}\left(
w\right) -f_{x}\left( \lambda u+\left( 1-\lambda \right) w\right)
\right\vert .
\end{eqnarray*}%
From Lemma 1, we have $\left( g-f\right) _{x}$ and $\left( g+f\right) _{x}$
are convex on the co-ordinates for all $\lambda \in \left[ 0,1\right] $ and $%
u,w\in \left[ c,d\right] .$ Similarly we can show that $\left( g-f\right)
_{y}$ and $\left( g+f\right) _{y}$ are also convex on the co-ordinates. This
completes proof.
\end{proof}

\begin{theorem}
Let $g:\Delta =\left[ a,b\right] \times \left[ c,d\right] \rightarrow R$ be
a co-ordinated convex mapping on $\Delta $ and $f:\Delta =\left[ a,b\right]
\times \left[ c,d\right] \rightarrow R$ is a co-ordinated $g-$%
convex-dominated mapping, where $a<b$ and $c<d$. Then, one has the
inequalities:%
\begin{eqnarray*}
&&\left\vert \frac{1}{\left( b-a\right) \left( d-c\right) }%
\int\limits_{a}^{b}\int\limits_{c}^{d}f\left( x,y\right) dydx-f\left( \frac{%
a+b}{2},\frac{c+d}{2}\right) \right\vert \\
&\leq &\frac{1}{\left( b-a\right) \left( d-c\right) }\int\limits_{a}^{b}\int%
\limits_{c}^{d}g\left( x,y\right) dydx-g\left( \frac{a+b}{2},\frac{c+d}{2}%
\right)
\end{eqnarray*}%
and%
\begin{eqnarray*}
&&\left\vert \frac{f\left( a,c\right) +f\left( a,d\right) +f\left(
b,c\right) +f\left( b,d\right) }{4}-\frac{1}{\left( b-a\right) \left(
d-c\right) }\int\limits_{a}^{b}\int\limits_{c}^{d}f\left( x,y\right)
dydx\right\vert \\
&\leq &\frac{g\left( a,c\right) +g\left( a,d\right) +g\left( b,c\right)
+g\left( b,d\right) }{4}-\frac{1}{\left( b-a\right) \left( d-c\right) }%
\int\limits_{a}^{b}\int\limits_{c}^{d}g\left( x,y\right) dydx.
\end{eqnarray*}
\end{theorem}

\begin{proof}
From Lemma 2, we can write%
\begin{eqnarray*}
&&\left\vert \frac{f\left( x,y\right) +f\left( z,w\right) }{2}-f\left( \frac{%
x+z}{2},\frac{y+w}{2}\right) \right\vert \\
&\leq &\frac{g\left( x,y\right) +g\left( z,w\right) }{2}-g\left( \frac{x+z}{2%
},\frac{y+w}{2}\right)
\end{eqnarray*}%
for all $\left( x,y\right) ,\left( z,w\right) \in \Delta .$ Set $%
x=ta+(1-t)b, $ $z=(1-t)a+tb,$ $y=tc+(1-t)d$ and $w=(1-t)c+td$ for all $%
t,s\in \left[ 0,1\right] .$ Then we get%
\begin{eqnarray*}
&&\left\vert \frac{f\left( ta+(1-t)b,tc+(1-t)d\right) +f\left(
(1-t)a+tb,(1-t)c+td\right) }{2}-f\left( \frac{a+b}{2},\frac{c+d}{2}\right)
\right\vert \\
&\leq &\frac{g\left( ta+(1-t)b,tc+(1-t)d\right) +g\left(
(1-t)a+tb,(1-t)c+td\right) }{2}-g\left( \frac{a+b}{2},\frac{c+d}{2}\right) .
\end{eqnarray*}%
By integrating with respect to $t$ over $\left[ 0,1\right] ^{2},$ we obtain
the desired result.%
\begin{eqnarray*}
&&\left\vert \frac{\int\limits_{0}^{1}\int\limits_{0}^{1}f\left(
ta+(1-t)b,tc+(1-t)d\right) dtds}{2}\right. \\
&&\left. +\frac{\int\limits_{0}^{1}\int\limits_{0}^{1}f\left(
(1-t)a+tb,(1-s)c+sd\right) dtds}{2}-f\left( \frac{a+b}{2},\frac{c+d}{2}%
\right) \right\vert \\
&\leq &\frac{\int\limits_{0}^{1}\int\limits_{0}^{1}g\left(
ta+(1-t)b,sc+(1-s)d\right) dtds}{2} \\
&&+\frac{\int\limits_{0}^{1}\int\limits_{0}^{1}g\left(
(1-t)a+tb,(1-s)c+sd\right) dtds}{2}.-g\left( \frac{a+b}{2},\frac{c+d}{2}%
\right)
\end{eqnarray*}%
Which completes the proof of the first inequality. From the definition of $%
g- $convex-dominated functions, we can write 
\begin{eqnarray*}
&&\left\vert tsf\left( a,c\right) +t(1-s)f\left( a,d\right)
+s(1-t)f(b,c)\right. \\
&&\left. +(1-t)(1-s)f(b,d)-f\left( ta+(1-t)b,sc+(1-s)d\right) \right\vert \\
&\leq &tsg\left( a,c\right) +t(1-s)g\left( a,d\right) +s(1-t)g(b,c) \\
&&+(1-t)(1-s)g(b,d)-g\left( ta+(1-t)b,sc+(1-s)d\right) .
\end{eqnarray*}%
By integrating with respect to $t$ over $\left[ 0,1\right] ^{2},$ we obtain%
\begin{eqnarray*}
&&\left\vert \frac{f\left( a,c\right) +f\left( a,d\right) +f\left(
b,c\right) +f\left( b,d\right) }{4}-\frac{1}{\left( b-a\right) \left(
d-c\right) }\int\limits_{a}^{b}\int\limits_{c}^{d}f\left( x,y\right)
dydx\right\vert \\
&\leq &\frac{g\left( a,c\right) +g\left( a,d\right) +g\left( b,c\right)
+g\left( b,d\right) }{4}-\frac{1}{\left( b-a\right) \left( d-c\right) }%
\int\limits_{a}^{b}\int\limits_{c}^{d}g\left( x,y\right) dydx.
\end{eqnarray*}%
Which completes the proof.
\end{proof}

We will prove some properties of the mapping $H$ in the following theorem.

\begin{theorem}
Let $g:\Delta =\left[ a,b\right] \times \left[ c,d\right] \rightarrow R$ be
a co-ordinated convex mapping on $\Delta $ and $f:\Delta =\left[ a,b\right]
\times \left[ c,d\right] \rightarrow R$ is a co-ordinated $g-$%
convex-dominated mapping, where $a<b$ and $c<d$. Then:

(i) $H_{f}$ is $H_{g}-$convex dominated on $[0,1]$.

(ii) One has the inequalities%
\begin{equation}
0\leq |H_{f}(t_{2},s_{2})-H_{f}(t_{1},s_{1})|\leq
H_{g}(t_{2},s_{2})-H_{g}(t_{1},s_{1})  \label{w}
\end{equation}

for all $0\leq t_{2}<t_{1}\leq 1$ and $0\leq s_{2}<s_{1}\leq 1.$

(iii) One has the inequalities%
\begin{equation*}
\left\vert f\left( \frac{a+b}{2},\frac{c+d}{2}\right) -H_{f}(t,s)\right\vert
\leq H_{g}(t,s)-g\left( \frac{a+b}{2},\frac{c+d}{2}\right)
\end{equation*}%
and%
\begin{eqnarray*}
&&\left\vert \frac{1}{\left( b-a\right) \left( d-c\right) }%
\int\limits_{a}^{b}\int\limits_{c}^{d}f\left( x,y\right)
dydx-H_{f}(t,s)\right\vert \\
&\leq &H_{g}(t,s)-\frac{1}{\left( b-a\right) \left( d-c\right) }%
\int\limits_{a}^{b}\int\limits_{c}^{d}g\left( x,y\right) dydx
\end{eqnarray*}%
for all $t,s\in \left[ 0,1\right] .$
\end{theorem}

\begin{proof}
(i) Since $f$ is co-ordinated $g-$convex-dominated mapping on $\Delta ,$ we
know that from Lemma 2, $g-f$ and $g+f$ are also convex on $\Delta $. By
using the linearity of the mapping $f\rightarrow H_{f},$ we can write $%
H_{g-f}=H_{g}-H_{f}$ and $H_{g+f}=H_{g}+H_{f}$. Then, it is easy to see that 
$H_{f}$ is $H_{g}-$convex dominated on $[0,1]$.

(ii) By Theorem 4, the mappings $H_{g-f}$ and $H_{g+f}$ are monotonic
nondecreasing on the co-ordinates. So, we have%
\begin{equation*}
H_{g}(t_{1},s_{1})-H_{f}(t_{1},s_{1})=H_{g-f}(t_{1},s_{1})\leq
H_{g-f}(t_{2},s_{2})=H_{g}(t_{2},s_{2})-H_{f}(t_{2},s_{2})
\end{equation*}%
and%
\begin{equation*}
H_{g}(t_{1},s_{1})+H_{f}(t_{1},s_{1})=H_{g+f}(t_{1},s_{1})\leq
H_{g+f}(t_{2},s_{2})=H_{g}(t_{2},s_{2})+H_{f}(t_{2},s_{2}).
\end{equation*}%
Then, we obtain%
\begin{equation*}
H_{f}(t_{2},s_{2})-H_{f}(t_{1},s_{1})\leq
H_{g}(t_{2},s_{2})-H_{g}(t_{1},s_{1})
\end{equation*}%
and%
\begin{equation*}
H_{f}(t_{1},s_{1})-H_{f}(t_{2},s_{2})\leq
H_{g}(t_{2},s_{2})-H_{g}(t_{1},s_{1}).
\end{equation*}%
Which is the desired result.

(iii) If we choose $t=s=0$ in (\ref{w}), we have the first inequality and by
a similar argument if we choose $t=s=1$, we obtain the second inequality.
\end{proof}

We will establish inequalities of Fejer-type for $g-$convex dominated
function on the co-ordinates in the following theorem.

\begin{theorem}
Let $g:\Delta \rightarrow 
\mathbb{R}
$ be a convex function on the co-ordinates and $f:\Delta \rightarrow 
\mathbb{R}
$ be a $g-$convex dominated function on the co-ordinates where $\Delta :=%
\left[ a,b\right] \times \left[ c,d\right] .$ Then the following
inequalities hold:%
\begin{eqnarray*}
&&\left\vert f\left( \frac{a+b}{2},\frac{c+d}{2}\right) -\frac{%
\dint\limits_{a}^{b}\dint\limits_{c}^{d}f(x,y)p\left( x,y\right) dydx}{%
\dint\limits_{a}^{b}\dint\limits_{c}^{d}p\left( x,y\right) dydx}\right\vert
\\
&\leq &g\left( \frac{a+b}{2},\frac{c+d}{2}\right) -\frac{\dint%
\limits_{a}^{b}\dint\limits_{c}^{d}g(x,y)p\left( x,y\right) dydx}{%
\dint\limits_{a}^{b}\dint\limits_{c}^{d}p\left( x,y\right) dydx}
\end{eqnarray*}%
and%
\begin{eqnarray*}
&&\left\vert \frac{f\left( a,c\right) +f\left( c,d\right) +f\left(
b,c\right) +f\left( b,d\right) }{4}-\frac{\dint\limits_{a}^{b}\dint%
\limits_{c}^{d}f\left( x,y\right) p\left( x,y\right) dydx}{%
\dint\limits_{a}^{b}\dint\limits_{c}^{d}p\left( x,y\right) dydx}\right\vert
\\
&\leq &\frac{g\left( a,c\right) +g\left( c,d\right) +g\left( b,c\right)
+g\left( b,d\right) }{4}-\frac{\dint\limits_{a}^{b}\dint\limits_{c}^{d}g%
\left( x,y\right) p\left( x,y\right) dydx}{\dint\limits_{a}^{b}\dint%
\limits_{c}^{d}p\left( x,y\right) dydx}
\end{eqnarray*}%
where $p:\left[ a,b\right] \times \left[ c,d\right] \rightarrow 
\mathbb{R}
$ is positive, integrable and symmetric about $x=\frac{a+b}{2}$ and $y=\frac{%
c+d}{2}.$
\end{theorem}

\begin{proof}
Since $f$ is $g-$convex dominated on the co-ordinates, from Lemma 2, the
mappings $f+g$ and $g-f$ are convex on the coordinates. By using Theorem 2,
we have:%
\begin{equation*}
\left( f+g\right) \left( \frac{a+b}{2},\frac{c+d}{2}\right) \leq \frac{%
\dint\limits_{a}^{b}\dint\limits_{c}^{d}\left( f+g\right) \left( x,y\right)
p\left( x,y\right) dydx}{\dint\limits_{a}^{b}\dint\limits_{c}^{d}p\left(
x,y\right) dydx}
\end{equation*}%
and%
\begin{eqnarray*}
&&\frac{\dint\limits_{a}^{b}\dint\limits_{c}^{d}\left( f+g\right) \left(
x,y\right) p\left( x,y\right) dydx}{\dint\limits_{a}^{b}\dint%
\limits_{c}^{d}p\left( x,y\right) dydx} \\
&\leq &\frac{\left( f+g\right) \left( a,c\right) +\left( f+g\right) \left(
c,d\right) +\left( f+g\right) \left( b,c\right) +\left( f+g\right) \left(
b,d\right) }{4}.
\end{eqnarray*}%
If we re-arrange these inequalities, we obtain:%
\begin{eqnarray*}
&&\dint\limits_{a}^{b}\dint\limits_{c}^{d}g\left( x,y\right) p\left(
x,y\right) dydx-g\left( \frac{a+b}{2},\frac{c+d}{2}\right) \\
&\leq &f\left( \frac{a+b}{2},\frac{c+d}{2}\right) -\frac{\dint%
\limits_{a}^{b}\dint\limits_{c}^{d}f(x,y)p\left( x,y\right) dydx}{%
\dint\limits_{a}^{b}\dint\limits_{c}^{d}p\left( x,y\right) dydx} \\
&\leq &g\left( \frac{a+b}{2},\frac{c+d}{2}\right) -\frac{\dint%
\limits_{a}^{b}\dint\limits_{c}^{d}g\left( x,y\right) p\left( x,y\right) dydx%
}{\dint\limits_{a}^{b}\dint\limits_{c}^{d}p\left( x,y\right) dydx}
\end{eqnarray*}%
and%
\begin{eqnarray*}
&&\frac{\dint\limits_{a}^{b}\dint\limits_{c}^{d}g\left( x,y\right) p\left(
x,y\right) dydx}{\dint\limits_{a}^{b}\dint\limits_{c}^{d}p\left( x,y\right)
dydx}-\frac{g\left( a,c\right) +g\left( c,d\right) +g\left( b,c\right)
+g\left( b,d\right) }{4} \\
&\leq &\frac{f\left( a,c\right) +f\left( c,d\right) +f\left( b,c\right)
+f\left( b,d\right) }{4}-\frac{\dint\limits_{a}^{b}\dint\limits_{c}^{d}f%
\left( x,y\right) p\left( x,y\right) dydx}{\dint\limits_{a}^{b}\dint%
\limits_{c}^{d}p\left( x,y\right) dydx} \\
&\leq &.\frac{g\left( a,c\right) +g\left( c,d\right) +g\left( b,c\right)
+g\left( b,d\right) }{4}-\frac{\dint\limits_{a}^{b}\dint\limits_{c}^{d}g%
\left( x,y\right) p\left( x,y\right) dydx}{\dint\limits_{a}^{b}\dint%
\limits_{c}^{d}p\left( x,y\right) dydx}.
\end{eqnarray*}%
Which completes the proof.
\end{proof}


\begin{thebibliography}{99}
\bibitem{1} S.S. Dragomir and N.M. Ionescu, On some inequalities for
convex-dominated functions, \textit{Anal. Num. Th\u{e}or. Approx.}, 19,
(1990), 21-28.

\bibitem{2} S.S. Dragomir, C.E.M. Pearce and J. Pe\v{c}ari\'{c}, Means, $g-$%
convex dominated functions \& Hadamard-type inequalities, \textit{Tamsui
Oxford Journal of Mathematical Sciences}, 18 (2), (2002), 161-173.

\bibitem{3} M.Z. Sar\i kaya, E. Set, M.E. \"{O}zdemir and S.S. Dragomir, New
Some Hadamard's type inequalities for co-ordinated convex functions,
Accepted.

\bibitem{4} M.E. \"{O}zdemir, E. Set, M.Z. Sar\i kaya, Some new Hadamard's
type inequalities for co-ordinated $m-$convex and ($\alpha ,m)-$convex
functions, Accepted.

\bibitem{h1} M.E. \"{O}zdemir, M. G\"{u}rb\"{u}z and H. Kavurmac\i ,
Hermite-Hadamard type inequalities for $(g,\varphi _{h})$ convex dominated
functions, arxiv.org (2012)

\bibitem{h2} M.E. \"{O}zdemir, M. Tun\c{c} and H. Kavurmac\i , Two new
differnt kinds of convex dominated functions and inequalities via
Hermite-Hadamard type, arxiv.org (2012)

\bibitem{h3} M.E. \"{O}zdemir, H. Kavurmac\i\ and M. Tun\c{c},
Hermite-Hadamard-type inequalities for new different kinds of convex
functions, arxiv.org (2012)

\bibitem{5} S.S. Dragomir, On the Hadamard's inequality for convex functions
on the co-ordinates in a rectangle from the plane, \textit{Taiwanese Journal
of Mathematics}, 5 (2001), no. 4, 775-788.

\bibitem{6} S.S. Dragomir, A mapping in connection to Hadamard's inequality,%
\textit{\ An Ostro. Akad. Wiss. Math. }-Natur (Wien) 128 (1991), 17-20. MR
93h: 26032.

\bibitem{7} S.S. Dragomir, Two mappings in connection to Hadamard's
inequality, \textit{J. Math. Anal, Appl.} 167 (1992), 49-56.

\bibitem{8} M. Alomari and M. Darus, Fejer inequality for double integrals,
Facta Universitatis, Ser. Math. Inform., 24 (2009), 15--28.
\end{thebibliography}
\end{document}